\documentclass[11pt, a4paper]{amsart}

\usepackage{amsmath,amssymb,amsfonts,mathrsfs,amsthm,bbm,enumitem, graphicx, mathabx, csquotes, color}

\newcommand{\ZZ}{{\mathbb{Z}}}

\newcommand{\HH}{{\mathbb{H}}}
\newcommand{\RR}{{\mathbb{R}}}
\newcommand{\PP}{{\mathbb{P}}}

\newcommand{\1}{\mathbbm{1}}

\newcommand{\A}{\mathscr{A}}
\newcommand{\Hit}{\mathsf{Hit}}
\newcommand{\PHit}{\mathsf{Hit^{\ast}}}
\newcommand{\Miss}{\mathsf{Miss}}
\newcommand{\PMiss}{\mathsf{Miss^{\ast}}}

\newcommand{\norm}[1]{\left\Vert #1\right\Vert}
\newcommand{\para}[1]{\left( #1 \right )}
\newcommand{\tub}[1]{\left\{#1\right\}}
\newcommand{\num}[1]{\left | #1\right |}
\newcommand{\kpara}[1]{\left[#1\right]}
\newcommand{\tB}[1]{B_p\para{0,#1}}

\newcommand{\nr}{\textup{SL}(n,\RR)}

\newcommand*\InsertTheoremBreak{
	\begingroup 
	\setlength\itemsep{0pt}
	\setlength\parsep{0pt}
	\item[\vbox{\null}]
	\endgroup
}

\newcommand{\LLL}{\mathscr{L}}

\DeclareMathOperator{\diag}{diag}

\newcommand{\SL}{\mathrm{SL}}

\newtheorem{theorem}{Theorem}[section]
\newtheorem{lemma}[theorem]{Lemma}
\newtheorem{proposition}[theorem]{Proposition}

\theoremstyle{definition}
\newtheorem{remark}[theorem]{Remark}

\usepackage{tikz}
\usetikzlibrary{arrows}

\pgfdeclarelayer{background}
\pgfsetlayers{background,main}

\newcommand{\E}{\mathscr{E}}

\usepackage[top=2cm,bottom=2cm,right=2cm,left=2cm]{geometry}
\title[On an extreme value law for the unipotent flow on $\SL_2(\RR)/\SL_2(\ZZ)$]{On an extreme value law for the unipotent flow on  $\SL_2(\RR)/\SL_2(\ZZ)$}

\author[Maxim Kirsebom]{Maxim Kirsebom}
\address{Universit\"at Hamburg, Bundesstrasse, 20146 Hamburg, Germany.}
\email{maxim.kirsebom@uni-hamburg.de }
\author[K. Mallahi-Karai]{Keivan Mallahi-Karai}
\address{Jacobs University Bremen, Campus Ring I, 28759 Bremen, Germany.}
\email{k.mallahikarai@jacobs-university.de }

\begin{document}

\begin{abstract} We study an extreme value distribution for the unipotent flow 
on the modular surface $\SL_2(\RR)/\SL_2(\ZZ)$. Using tools from homogenous dynamics and geometry of numbers we prove the existence of a continuous distribution function $F(r)$ for the normalized deepest cusp excursions of the unipotent flow. We find closed analytic formulas for $F(r)$ for 
$r \in [-\frac{1}{2} \log 2, \infty)$, and establish asymptotic behavior of $F(r)$ as $r \to -\infty$. 
\end{abstract}

\maketitle

\section{Introduction}

In this paper we consider the action of the unipotent flow on the space of unimodular lattices in $\RR^2$ which we will denote by $\LLL_2$. $\LLL_2$ can be viewed as the homogeneous space $\SL_2(\RR)/\SL_2(\ZZ)$ and at the same time the unit tangent bundle of $V_1:=\HH^2/\SL_2(\ZZ)$. This allows for two geometric interpretations of the unipotent flow. For the purpose of explaining our question of interest we begin with the picture in the upper-half plain as illustrated in Figure \ref{Fig-1}. Here the unipotent flow $u_t$, also known as the horocycle flow, acts on a pair $(z,v)\in \textup{T}^1\HH^2$ by moving $z$ a distance $t$ clockwise along the unique horocycle which intersects the real line in a single point, goes through $z$ and whose tangent in $z$ is orthogonal to $v$. For $(z,v)\in \textup{T}^1\HH^2$ we will refer to $z$ as the base point. The region $F$ in Figure \ref{Fig-1} and $\textup{T}^1 F$ are fundamental domains, respectively, for the actions of 
$\SL_2(\ZZ)$ of  $ \HH^2$ and $\textup{T}^1 \HH^2$.  In $\textup{T}^1 F$, vectors with base on the vertical sides are identified via the transformation $(z,v)\mapsto(z+1,v)$ and vectors with base on the right and left arc segments meeting in $i$ are identified via the transformation $(z,v)\mapsto \para{-\frac{1}{z},\frac{v}{z^2}}$.
\begin{figure}[h]
	\centering
	\includegraphics[width=9cm]{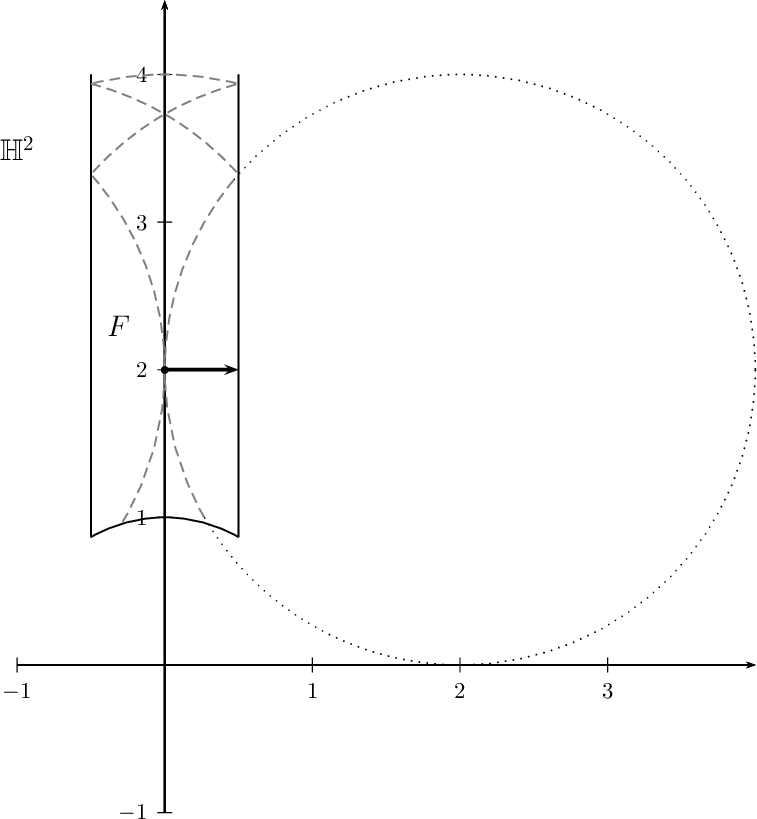}
	\caption{Sketch of an orbit of the Horocycle flow in $\textup{T}^1 F$.}
	\label{Fig-1}
\end{figure}
The unipotent flow on $\textup{T}^1 V_1$ is mixing and hence ergodic with respect to the Liouville measure. 
The proof of ergodicity goes back to Hedlund \cite{Hedlund1}, while the mixing property is due to Parasyuk \cite{Parasyuk}. It follows that almost every orbit of the unipotent flow is dense in $\textup{T}^1 V_1$, which in turn implies that the orbit of almost every point $(z,v)\in\textup{T}^1 F$ will make deeper and deeper excursions into the cusp, where the cusp refers to the infinite section of $F$ in the vertical direction. There is a natural interest in understanding various aspects of the deepest excursion into the cusp up to some given time. We begin by reviewing known results in this direction.

\subsection{Known results}

Unsurprisingly, early studies of cusp excursions focused on the geodesic flow as opposed to the unipotent flow. The fast mixing of the geodesic flow provides a strategy for how to approach the problem which is not available in the unipotent case. Sullivan \cite{Sull} proved a so-called logarithm law for the geodesic flow on the unit tangent bundle of $V_{d}:=\HH^{d+1}/\Gamma$ where $\Gamma$ is a discrete subgroup of isometries of $\HH^{d+1}$ such that $V_{d}$ is of finite volume but not compact. More precisely, denote by $\textup{dist}$ the hyperbolic distance on $V$ and by 
$\gamma_t$  the geodesic flow on $T^1 V_d$. Write $\pi: T^1 V_d \to V_d$ for the map sending $(z, v)$ to $z$. 
Sullivan proved that for all $z \in V_d$ and almost all unit tangent vectors
$v$ at $z$, we have
\begin{equation*}
	\limsup_{ t \to \infty}\frac{\textup{dist} (\pi \gamma_t (z, v), z_0) )}{\log t}=\frac{1}{d},
\end{equation*}
where $ z_0 $ is an arbitrary base point on $V_d$. 
Kleinbock and Margulis \cite{KleinMarg} later extended this result to more general flows on homogeneous spaces with more general observables than the distance function. A particularly interesting case for which their results hold is for a certain $\RR$-diagonalizable 1-parameter subgroup of $G:=\SL_n(\RR)$ acting on the homogeneous space $X:=\SL_n(\RR)/\SL_n(\ZZ)$. The result is best formulated using the identification of $X$ with the space of unimodular lattices $\LLL_n$ which we already mentioned in the case of $n=2$. Via this identification, the action of $G$ on $X$ corresponds to the action of $G$ on lattices induced from the standard action of $G$ on $\RR^n$ by matrix multiplication. This action preserves a probability measure $\mu_n$ on $\LLL_n$ which is induced by the Haar measure on $G$.
Kleinbock and Margulis considered the 1-parameter subgroup
\begin{equation*}
a_t:=\tub{\diag\para{e^{\frac{t}{m}},\dots,e^{\frac{t}{m}},e^{\frac{t}{k}},\dots,e^{\frac{t}{k}}}:t\in\RR},
\end{equation*}
where $m+k=n$, and the observable $\alpha_1:\LLL_n\to\RR_+$ defined by
\begin{equation}
\label{alpha1}
\alpha_1(\Lambda):= \sup_{0\neq v\in\Lambda} \frac{1}{\norm{v}}.
\end{equation}
Note that Mahler's compactness criterion asserts that a sequence $ \Lambda_i$ of lattices in $\LLL_n$ diverges to infinity if and only if $ \alpha_1( \Lambda_i) \to \infty$. Kleinbock and Margulis showed that for $\mu_n$-almost every $\Lambda\in\LLL_n$ we have
\begin{equation}\label{KMLogLaw}
 \limsup_{ t \to \infty} \frac{\log \alpha_1( a_t \Lambda)}{\log t} = \frac{1}{n}.
\end{equation}

In \cite{AtMa1}, Athreya and Margulis generalized the logarithm law to unipotent flows on $\LLL_n$. More precisely they proved that \eqref{KMLogLaw} holds $\mu_n$-almost surely when $a_t$ is replaced by any unipotent 1-parameter subgroup $\tub{u_t}_{t\in\RR}$ of $\nr$ for $n \geq 2$. 
In the special case of $n=2$, Athreya and Margulis proved an even stronger result for the lower bound in the sense that for every $\Lambda\in\LLL_2$ such that $\tub{u_t\Lambda}_{t\in\RR}$ is not periodic we have 
\begin{equation*} 
\limsup_{ t \to \infty} \frac{\log \alpha_1( u_t \Lambda)}{\log t} \ge \frac{1}{2}.
\end{equation*}
For systems obeying a logarithm law, more subtle differences in cusp excursion behaviour among different orbits are possible. For example, orbits may outperform, or be outperformed by, the expected asymptotic behaviour, i.e. logarithm of the time, by a fixed additive amount. Such behaviour is not detectable in a logarithm law, but can be studied through a so-called extreme value law (EVL) -- if one such holds. For the geodesic flow on $T^1V_1$, Pollicott \cite{MarkPollicott} proved the following extreme value law. Let $\mu_L$ denote the Liouville measure on $T^1V_1$, let $h$ denote the function returning the hyperbolic height above the horizontal line $\textup{Im}(z)=1$ and let $\gamma_t$ denote the geodesic flow on $T^1V_1$. Pollicott proved that for any $r\in\RR$,
	\begin{equation*}
	\lim_{T\to\infty} \mu_L\tub{(z,v)\in T^1V_1:\max_{0\leq t\leq T} h(\gamma_t(z,v)) \leq r+\log T}=e^{-\frac{3}{\pi^2}e^{-r}}.
	\end{equation*}
The proof of this theorem is based on the connection 
of the geodesic flow on $T^1V_1$ to the continued
fraction expansion, for which an extreme value theorem 
was established by Galambos \cite{Galambos}. 

This result implies the logarithm law for the geodesic flow on $T_1V_1$, hence a result of this kind is a generalization of the logarithm law. The distribution on the right hand side is known as a Gumbel distribution. 

Generally, extreme value laws are studied in probability in the field of extreme value theory (EVT). In a general probabilistic setting one considers a probability space $(\Omega, \PP)$ as well as random variables $\xi_t:\Omega\to\RR$ for $t$ belonging to some index set which may be either discrete or not. In the classical (discrete) case one defines the random variable
\begin{equation*}
	M_n:=\max_{0\leq i\leq n} \para{\xi_n}
\end{equation*} 
and asks whether we can find real-valued functions $a_n>0$ and $b_n$ such that the limit
\begin{equation*}
	\lim_{n\to\infty}\PP\tub{M_n\leq a_n r+b_n}=:G(r)
\end{equation*} 
exists, and if so, what the limit is. For independent and identically distributed random variables, a central result of EVT is that only three different distributions, known as Gumbel, Frechet and Weibull distributions, may appear as $G(r)$ above. Furthermore, this result also holds for certain weakenings of the independence assumption on $\xi_t$, including various mixing assumptions. These results in turn generalize to stochastic processes with continuous parameter. See \cite{LLR} for a reference to EVT.

The main aim of this paper is to investigate the existence and potential form of an EVL for maximal excursions of the unipotent flow on $\LLL_2$ where maximal is meant with respect to the observable $\log \alpha_1$.

\subsection{Main result}

We retain the notation of the previous section and consider the unipotent flow on $\LLL_2$, which is given by the following 
one-parameter family of transformations:
\begin{equation*}
	u_t=  \begin{pmatrix}
		1  & t    \\
		0 & 1 \\
	\end{pmatrix}, \qquad  t \in \RR.
\end{equation*}

For simplicity we write $\mu:=\mu_2$. As mentoned above, we are interested in the existence and form of an extreme value law for the unipotent flow in this setting. Let 
\begin{equation} \label{ET}
\E_T(r):=\tub{\Lambda\in\LLL_2:\max_{0\leq t\leq T} \log\alpha_1(u_t\Lambda)\leq r+\frac12\log T}.
\end{equation}
In words, we consider those lattices whose $u_t$-orbit over the interval $t \in [0, T]$ fails to outperform the expected asymptotic behavior by an additive amount $r$.

Our main result is the following.
\begin{theorem}\label{main1} Let $\E_T(r)$ be defined as in \eqref{ET}. The limit 
	\begin{equation*}
		F(r):=\lim_{ T \to \infty} \mu(\E_T(r))
	\end{equation*}
	exists and defines a  continuous function of $r$ for all $r\in\RR$. Furthermore:
\begin{enumerate}
\item For all $r>0$,
\begin{equation*}
	F(r)= 1 - \frac{3}{\pi^2}e^{-2r}.
\end{equation*}
\item For all $-\frac12\log 2< r\leq 0$,
\begin{equation*}
	F(r)= 1 - 
	\frac{3}{\pi^2}\para{-e^{-2r}+4r^2-4r+2}.
\end{equation*}
\item\label{Main3} There exist positive constants $C_0, C_1$  such that for all $r \le -\frac12 \log 2$
\[  C_0 e^{2r} \le     F(r) \le C_1  e^{2r}. \]
\end{enumerate}	
	\end{theorem}
Set
\begin{align*}
	F_1(r):= 1 - \frac{3}{\pi^2}e^{-2r}\quad\text{and}\quad F_2(r):= 1 - \frac{3}{\pi^2}\para{-e^{-2r}+4r^2-4r+2}.
\end{align*}
$F_1$ and $F_2$ are depicted in Figure \ref{Fig0} below, the solid parts representing the known part of the EVL for the unipotent flow, while the dotted parts represent the continuations of $F_1$ and $F_2$ respectively. 
\begin{figure}[h!]
	\centering
	\includegraphics[width=12cm]{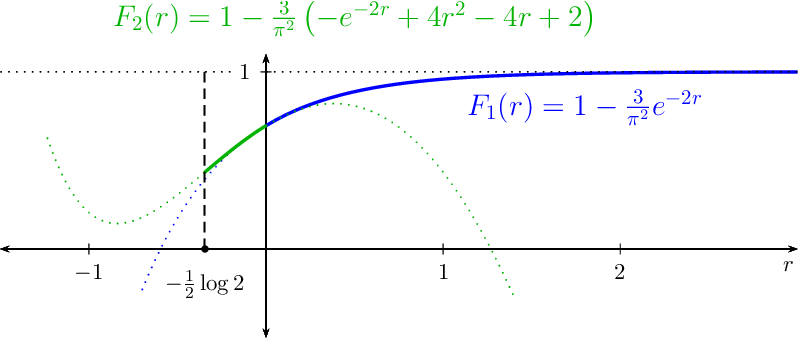}
	\caption{The known part of $F(r)$.}
	\label{Fig0}
\end{figure}
Unfortunately, at this stage our methods appear insufficient to extend the range of the parameter $r$ for which an explicit form of the EVL can be given.

Given the differences between the geodesic flow and the unipotent flow, in particular their quantitative mixing properties\footnote{See for example \cite{Ratner}.}, one would expect the two flows to exhibit different behaviour regarding cusp excursions. It it worth noting that such difference is not revealed through logarithm laws, but instead only through the more precise questions of EVL's. This provides a strong motivation for the study of EVL's in dynamical systems. 

\subsection{Structure of the paper}
in Section \ref{SectionLattices} we give a quick review of some well-known results concerning integration over the space of unimodular lattices as well as some easy consequences of said results. For generality, these are stated for unimodular lattices in $\RR^n$ despite only being applied for $n=2$. These integration formulas will play a central role in the proof of Theorem \ref{main1}. In Section \ref{MainProofPart1} we give the proof of Theorem \ref{main1}.
 
\subsection{notation}
The characteristic function of a set $B$ will be denoted by $\1_B$. The Lebesgue measure on $\RR^2$ will be denoted by $m$.  We will use
the notation $ \{ \Lambda: ... \} $ as a shorthand for 
$ \{ \Lambda \in \LLL_2: ... \} $.

\section{Integration over the space of lattices}\label{SectionLattices}

Let $f:\RR^n\to\RR$ be a function of bounded support. Then the Siegel transform of $f$ is the function $\widehat{f}:\LLL_n\to\RR$ defined by
\begin{equation*}
\widehat{f}(\Lambda):=\sum_{v\in\Lambda\backslash\tub{0}} f(v).
\end{equation*}  
A very useful aspect of the Siegel transform is the related Siegel formula relating the integral of $\widehat{f}$ to the integral of $f$ \cite[Theorem 2]{Siegel2}. It states that if $f$ is measurable and non-negative, then
\begin{equation*}
\tag{Siegel's Formula}
\int_{\LLL_n} \widehat{f}(\Lambda)\,d\mu =\int_{\RR^n} f(v)\,d(v).
\end{equation*}
The following lemma is an application of Siegel's formula.
\begin{lemma}\label{hitting}
	Let $B \subseteq \RR^n$ be measurable and $0 \not\in B$. Then we have 
	\begin{equation*}
	\mu \{ \Lambda \in \LLL_n: \Lambda \cap B \neq \emptyset \} \le m(B).
	\end{equation*}
	\begin{proof}
		Apply Siegel's formula to $\1_B$ and invoke the inequality
		\begin{equation*}
		\sum_{v\in\Lambda\backslash\tub{0}} \1_B(v)\geq \1_{\tub{\Lambda\in\LLL_n : \Lambda\cap B\neq \emptyset}}.
		\end{equation*} 
	\end{proof}
\end{lemma}

\begin{remark}
We note that in \cite{AtMa1} Athreya and Margulis proved that for every $n \ge 2$, there exists a constant $C_n$ such that for every measurable set $B \subseteq \RR^n$, we have 
\begin{equation*}
\mu \{ \Lambda: \Lambda \cap B =\emptyset \} \le \frac{C_n}{m(B)}.
\end{equation*}
This inequality is particularly useful when $m(B)$ is large, and provides an upper bound on the probability of missing a set of large measure. Lemma \ref{hitting} can be seen as a \enquote{hitting} counterpart of this fact, proving a similar upper bound for the probability of hitting a set of small measure. 
\end{remark}

We will also use a variant of the Siegel transform which is useful in dealing with primitive points. Recall that a non-zero point
$v \in \Lambda$ is called primitive, if it cannot be written as an integer multiple of another point in $\Lambda$. The subset of $\Lambda$ consisting of its primitive points is denoted by $\Lambda^{\ast}$. The primitive Siegel transform of $f$ is defined by 
\begin{equation*}
\widecheck{ f }( \Lambda) = \sum_{v \in \Lambda^{\ast}} f(v).
\end{equation*} 
For the primitive Siegel transform we have an analogue of Siegel's Formula which states
\begin{equation}
\label{ModifiedSiegelFormula}
\int_{\LLL_n}\widecheck{ f }( \Lambda) \,d\mu(\Lambda)=\frac{1}{\zeta(n)}\int_{\RR^n} f(v)\,d(v).
\end{equation}
where $\zeta(n)$ denotes the well-known zeta function, see \cite[(25)]{Siegel2}.
We will need the following lemma.

 \begin{lemma}\label{SLnRinv}
	For any measurable set $A\subset \RR^n$ and for all $g\in \SL_n(\RR)$
	\begin{equation*}
	\mu(\tub{\Lambda\in\LLL_n:\Lambda\cap A\neq \emptyset})=\mu(\tub{\Lambda\in \LLL_n:\Lambda \cap gA\neq \emptyset}).
	\end{equation*}
	\begin{proof}
		Let $R(A)=\tub{\Lambda\in\LLL_n : \Lambda \cap A \neq \emptyset}$. It is straightforward that 
		\begin{align*}
		gR(A)=\tub{g\Lambda\in\LLL_n : \Lambda \cap A \neq \emptyset}
		=\tub{\Lambda'\in\LLL_n : \Lambda' \cap gA \neq \emptyset}
		= R(gA).
		\end{align*}
		The conclusion then follows since $\mu$ is $\SL_n(\RR)$-invariant.
	\end{proof}
\end{lemma}

For part (\ref{Main3}) of Theorem \ref{main1} the following two results will be important.
Recall that the successive minima $ \lambda_1( \Lambda) \le \cdots \le \lambda_n( \Lambda)$ of $ \Lambda$ are defined as follows: For $ 1 \le i \le n$, denote by $ \lambda_i( \Lambda)$ the least $r$ such that 
$ \Lambda$ contains $i$ linearly independent vectors of Euclidean norm at most $r$. The following is a result of Minkowski.
\begin{lemma}\cite[Theorem 16]{Siegel}\label{Minkowski}
	for any $ \Lambda \in \LLL_n$ we have
	\begin{equation}\label{product}
	\lambda_1( \Lambda) \cdots   \lambda_n( \Lambda) \le \frac{2^n}{V_n},
	\end{equation}
	where $V_n$ denotes the volume of the unit ball in $\RR^n$.
\end{lemma}
The following is a special case of a lemma by Schmidt.
\begin{lemma}\cite[Lemma 1]{Schmidt2}\label{Davenport}
	Let $S$ be a compact convex subset of $\RR^n$ which lies in a ball of radius $R$ centered at zero. 
	For any unimodular lattice $ \Lambda$ in $\RR^n$ if $ \lambda_{n-1} \le R$ we have 
	\[ \left|  |\Lambda \cap S| - {\mathrm{vol} }(S)  \right|  \le C_n \lambda_n( \Lambda) R^{n-1}. \]
	for some constant $C_n>0$ depending only on the dimension $n$. 
\end{lemma}

\section{Hitting probabilities and the proof of Theorem \ref{main1} part (1)}\label{MainProofPart1}

The proof naturally splits into three parts. First we simplify the problem and relate the cumulative distribution function $F(r)$ to a problem in geometry of numbers.  This step will be used for the proof of all three parts of the theorem.

\subsection{Reformulation and simplification of the problem}

We begin by rewriting the event $\E_T(r)$. In the following, set $H_R:=\tub{(x,y)\in \RR^2:x^2 + y^2 \leq R^2, y\geq 0}$, i.e. the upper half of the disk of radius $R$ centered at the origin. 
\begin{align*}
	\E_T(r)&=\tub{\Lambda: \max_{0\leq t\leq T} \log\alpha_1(u_t\Lambda)\leq r+\frac12\log T}\\
		&=\tub{\Lambda :\max_{0\leq t\leq T} \sup_{0\neq v\in u_t\Lambda}\frac{1}{\norm{v}}\leq e^rT^{\frac12}}\\
		&=\tub{\Lambda :\min_{0\leq t\leq T} \inf_{0\neq v\in u_t\Lambda}\norm{v}\geq e^{-r}T^{-\frac12}}\\
		&\stackrel{(*)}{=}\tub{\Lambda :\bigcup_{0\leq t\leq T} u_t\Lambda \cap H_{e^{-r}T^{-\frac12}}=\tub{0}}\\
		&=\tub{\Lambda:\Lambda \cap \bigcup_{0\leq t\leq T} u_{-t}H_{e^{-r}T^{-\frac12}}=\tub{0}}\\
		&=\tub{\Lambda:\Lambda \cap \para{\bigcup_{0\leq t\leq T} u_{-t}H_{e^{-r}T^{-\frac12}}\backslash\tub{0}}=\emptyset}.
\end{align*}
The $(*)$-equality follows simply from the fact that $v\in\Lambda \iff -v\in\Lambda$ and hence $\Lambda\cap H_R=\tub{0}$ iff 
$\Lambda$ contains no non-zero vectors in the disk of radius $R$. 
Excluding the 0-vector in the last line is simply a matter of notational convenience. 

In the future we will define similar sets, the zero vector has to be removed for similar reasons. 
Henceforth we will write
\begin{equation*}
	D_{r,T}:=\bigcup_{0\leq t\leq T} u_{-t}H_{{e^{-r}T^{-\frac12}}}\backslash\tub{0},
\end{equation*}
which means that
\begin{equation}\label{ETr}
	\E_T(r)=\tub{\Lambda:\Lambda \cap D_{r,T}=\emptyset}.
\end{equation}
We see that the measure of $\E_T(r)$ corresponds to the probability that a random lattice does not intersect $D_{r,T}$. 
It will be convenient to replace $D_{r,T}$ with simpler sets approximating it from the inside and outside. Define
\begin{align*}
	I_R&:=\tub{0}\times \kpara{0, R}\\
	O_R&:=\kpara{-R,R}\times \kpara{0,R}.
\end{align*}
Clearly $I_{R}\subset H_{R} \subset O_{R}$. The sets are depicted in Figure \ref{Fig1}. 
\begin{figure}[h]
	\centering
	\includegraphics[width=10cm]{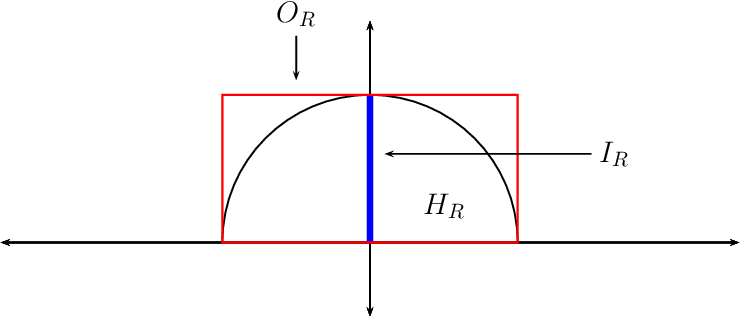}
	\caption{}
	\label{Fig1}
\end{figure}

By setting 
\begin{align*}
	D^-_{r,T}&:=\bigcup_{0\leq t\leq T} u_{-t}I_{e^{-r}T^{-\frac12}}\backslash\tub{0}\\
	D^+_{r,T}&:=\bigcup_{0\leq t\leq T} u_{-t}O_{e^{-r}T^{-\frac12}}\backslash\tub{0}.
\end{align*}
we obtain the inclusion $D^-_{r,T}\subset D_{r,T}\subset D^+_{r,T}$. The next lemma concerns $D^-_{r,T}$ and $D^+_{r,T}$.

\begin{lemma}\label{approx} 
	\InsertTheoremBreak
\begin{enumerate}
\item $D^-_{r,T}$ is the triangle with vertices $(0,0)$, $(0, e^{-r}T^{-\frac12})$ and $(-e^{-r}T^{\frac12}, e^{-r}T^{-\frac12})$ with the point $(0,0)$ removed.

\item for $T \gg 1$, we have 
 \begin{equation*}\label{comparison}
 m\para{D^+_{r,T}\backslash D^-_{r,T}} \le    2e^{-2r} T^{-1}.
\end{equation*}
\end{enumerate}
\begin{figure}[h]
	\centering
	\includegraphics[width=14cm]{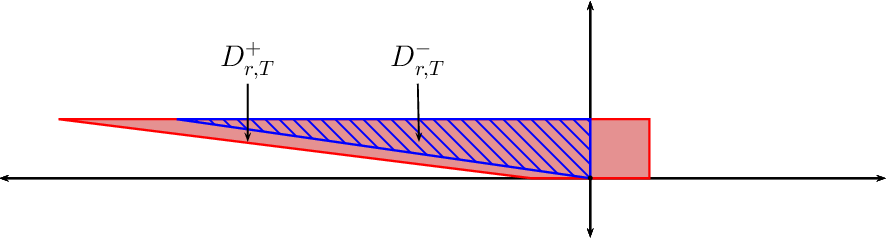}
	\caption{}
	\label{Fig2}
\end{figure}
\begin{proof}	\InsertTheoremBreak
\begin{enumerate}
\item
Note that the action of $u_{-t}$ is given by 
\begin{equation*} \begin{pmatrix}
1  & -t    \\
 0 & 1 \\
\end{pmatrix} \begin{pmatrix} x    \\
 y \\
\end{pmatrix} = \begin{pmatrix} x-ty    \\
 y \\
\end{pmatrix}
\end{equation*}
Hence $u_{-t}$ moves points horizontally by an amount which depends linearly on the $y$-coordinate of the point. In particular, the line $I_{e^{-r}T^{-\frac12}}$ is mapped to the line with endpoints $(0,0)$ and $(-te^{-r}T^{-\frac12},e^{-r}T^{-\frac12})$ under $u_t$. Taking the union over $0\leq t\leq T$ gives the claim.

\item From Figure \ref{Fig2} we see that $D^+_{r,T}\backslash D^-_{r,T}$ consists of two disjoint sets, a parallelogram on the left and a square on the right. The area of the square equals $e^{-2r}T^{-1}$ and a simple geometric argument shows that the area of the parallelogram is bounded by $e^{-2r}T^{-2}$. Hence the postulated bound is true. 
\end{enumerate}
\end{proof}
\end{lemma}

For simplicity we introduce the following notation. For a measurable subset $A \subseteq \RR^2$, we define the events 
$\PHit(A)\subset \Hit (A) \subset \LLL_2$ by 
\begin{align*}
\Hit(A)&:= \{ \Lambda: \Lambda \cap A \neq \emptyset \}. \\
\PHit(A)&:= \{ \Lambda: \Lambda^{\ast} \cap A \neq \emptyset \}.
\end{align*} 

\begin{proposition}\label{hit} For measurable sets $A, B, C \subseteq \RR^2$ we have 
\begin{enumerate}
\item \label{Hit1} If $ A \subseteq B \cup C$ then $\Hit(A) \subseteq \Hit(B) \cup \Hit(C)$.
\item \label{Hit2} Assume that $A$ is convex such that $0 \not\in A$ but $0$ is an accumulation point of $A$. Then $\Hit(A)=\PHit(A)$.
\end{enumerate}
\end{proposition}

\begin{proof}
(1) is clear. For (2) The $\supset$ direction is trivial since $\Lambda^{\ast}\subset \Lambda$. For the $\subset$ direction, if $v\in \Lambda\backslash \Lambda^{\ast}$ and $v\in A$, then since $v \neq 0$, we can express
 $v=kv^{\ast}$ for some $v^{\ast}\in \Lambda^{\ast}$, $k>1$. Since $A$ is convex and $0$ is an accumulation point of $A$ we conclude that $v^{\ast}\in A$.
\end{proof}

We will also denote
\begin{equation*}
\Delta_r:=\tub{(x,y) \in \RR^2: 0\leq x\leq e^{-r}, y\leq x}\backslash \tub{0},
\end{equation*}
i.e. the triangle with vertices $(0,0), (e^{-r},0), (e^{-r},e^{-r})$ with the zero-vector removed, see Figure \ref{Fig3}.

\begin{lemma}\label{lemma:limit}
 For $r \in \RR$, the limit 
 $ F(r):= \lim_{T\to\infty}\mu(\E_T(r))$ exists and is given by 
\begin{equation}
\label{Dr1limit}
	F(r)=1-\mu( \Hit ( \Delta_r) )=1-\mu(\PHit(\Delta_r)).
\end{equation}
\end{lemma}

\begin{proof}
We will start by proving the first equality. 
First we show that
\begin{equation}
\label{DrTDr1}
\mu(\Hit ( D^-_{r,T})) =\mu ( \Hit (  \Delta_{r}) ). 
\end{equation}

Let
\begin{align*}
a_T:=\begin{pmatrix}
T^{-\frac12} & 0 \\
0 & T^{\frac12}
\end{pmatrix}\in\SL_2(\RR).\\
k_{\theta}:=\begin{pmatrix}
\cos\theta & -\sin\theta \\
\sin\theta & \cos\theta
\end{pmatrix}\in\SL_2(\RR).
\end{align*}
We consider the effect of applying $a_T$ to $D^-_{r,T}$. Since $a_T$ is a continuous bijection and $D^-_{r,T}$ is compact, its boundary is mapped to the boundary of $a_T D^-_{r,T}$. Furthermore, straight lines are mapped to straight lines under $a_T$, hence $a_T D^-_{r,T}$ is also a triangle and its vertices are the images of the vertices of $D^-_{r,T}$ under $a_T$. Using the vertices of $D^-_{r,T}$ found in Lemma \ref{approx} we get 
\begin{equation*}
a_T\begin{pmatrix}
0\\
0
\end{pmatrix}=\begin{pmatrix}
0\\
0
\end{pmatrix},\quad 	
a_T\begin{pmatrix}
0\\
e^{-r}T^{-\frac12}
\end{pmatrix}=\begin{pmatrix}
0\\
e^{-r}
\end{pmatrix},\quad 	
a_T\begin{pmatrix}
-e^{-r}T^{\frac12}\\
e^{-r}T^{-\frac12}
\end{pmatrix}=\begin{pmatrix}
-e^{-r}\\
e^{-r}
\end{pmatrix}.
\end{equation*}
It is well known that $k_{\frac{3\pi}{2}}$ correspond to a clockwise rotation by $\frac{\pi}{2}$, hence $k_{\frac{3\pi}{2}}a_TD_{r,T}^-$ is the triangle with vertices $(0,0)$, $(e^{-r},0)$ and $(e^{-r},e^{-r})$.
Since $D^-_{r,T}$ was defined as not containing the point $(0,0)$ we conclude that
\begin{equation*}
\Delta_{r}=k_{\frac{3\pi}{2}}a_T D^-_{r,T}.
\end{equation*}
Using Lemma \ref{SLnRinv} we get \eqref{DrTDr1}. Trivially we also conclude that
\begin{equation*}
	\lim_{T\to\infty}\mu ( \Hit( D^-_{r,T})) =\mu ( \Hit( \Delta_{r})),
\end{equation*}
since the right hand side does not depend on $T$. So we are left with arguing that 
\begin{equation}
\label{DrTDrT-}
	\lim_{T\to\infty}\mu ( \Hit (  D_{r,T})) =\lim_{T\to\infty}\mu ( \Hit(  D^-_{r,T}) ).
\end{equation}
Recall that $D^-_{r,T}\subset D_{r,T}\subset D^+_{r,T}$ and therefore
\begin{equation*}
	D^-_{r,T}\subset D_{r,T}=D^-_{r,T}\cup D_{r,T}\backslash D^-_{r,T}\subset D^-_{r,T}\cup D^+_{r,T}\backslash D^-_{r,T}.
\end{equation*}
Consequently,
\begin{gather}
\begin{split}
\label{approx2}
	\mu(\Hit( D^-_{r,T} )) \leq \mu(\Hit( D_{r,T}) ) &\leq \mu(\Hit( D^-_{r,T}) )+ \mu(\Hit ( D^+_{r,T}\backslash D^-_{r,T} )) \\
	&\leq \mu(\Hit(  D^-_{r,T}))+2e^{-2r} T^{-1}.
\end{split}
\end{gather}
Where the last inequality follows from Lemma \ref{hitting} in conjunction with Lemma \ref{approx}.
Taking limits in \eqref{approx2} gives \eqref{DrTDrT-} which in turn implies \eqref{Dr1limit} as claimed.
By Proposition \ref{hit}, $\Hit(\Delta_r)=\PHit(\Delta_r)$, from which the second equality follows. 
\end{proof}
\begin{figure}[h]
	\centering
	\includegraphics[width=5cm]{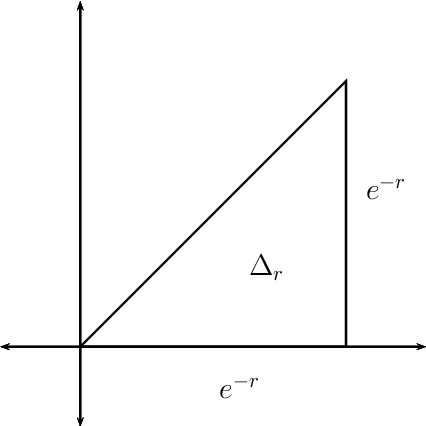}
	\caption{}
	\label{Fig3}
\end{figure}

\begin{remark} Similar connections between asymptotic distribution problems and hitting probabilities in geometry of numbers have already appeared in the literature.  
In \cite{EM}, Elkies and McMullen relate the question of gap distribution of the 
sequence $ \{ \sqrt{n} \pmod{1}: n=1,2, \dots \} $ to the probability that a random {\it affine} lattice in $\RR^2$ hits a triangle of a given area. The distribution function obtained in \cite{EM} 
is a piecewise defined function, with different formulas over three disjoint intervals.  
\end{remark}

\subsection{Proof of continuity of $F(r)$}
By \eqref{Dr1limit} it suffices to argue that $\mu(\Hit(\Delta_r))$ is continuous. For $\delta>0$, we compute
\begin{align*}
	\mu(\Hit(\Delta_{r+\delta}))-\mu(\Hit(\Delta_{r}))&=\mu(\Hit(\Delta_{r+\delta})\backslash \Hit(\Delta_{r}))\\
	&=\mu(\Hit(\Delta_{r+\delta}\backslash \Delta_{r}))\\
	&\leq m(\Delta_{r+\delta}\backslash \Delta_{r}),
\end{align*}
where we used Lemma \ref{hitting} to obtain the last inequality. The set $\Delta_{r+\delta}\backslash \Delta_{r}$ is a trapezoid which is easily seen to have area $\delta e^{-r}+\delta^2\to 0$ as $\delta\to 0$. This proves right continuity of $F(r)$, left continuity follows analogously.

\subsection{Proof of Theorem \ref{main1} part (1)}\label{MainProofFirstEnd}
 The primitive Siegel transform of $\1_{\Delta_r}$ is given by
 \begin{equation}\label{eq:ModCharDelta}
 \widecheck{ \1_{\Delta_r} }( \Lambda) = \sum_{v \in \Lambda^{\ast}} \1_{\Delta_r}(v)=\sum_{n=0}^{\infty}n\1_{\tub{\Lambda:\num{\Lambda^{\ast} \cap \Delta_r}=n}}.
 \end{equation}
 Note that for $r>0$ we have $e^{-r}<1$ which implies that the triangle $\Delta_r$ has area less than $1/2$. This means that any unimodular lattice can have at most one primitive point inside $\Delta_r$, since any two points inside $\Delta_r$ span a parallelogram of area less than $1$.
 Hence all terms in the last sum of \eqref{eq:ModCharDelta} vanish except for  the term corresponding for $n=1$. Hence 
\begin{equation*}
	\widecheck{ \1_{\Delta_r} }( \Lambda) =\1_{\PHit(\Delta_r)}.
\end{equation*} 
Applying equality \eqref{ModifiedSiegelFormula} to $\widecheck{ \1_{\Delta_r} }$ we obtain
\begin{equation} \label{firstmoment}
	\mu(\PHit(\Delta_r))
	=\frac{1}{\zeta(2)}\int_{\RR^2} \1_{\Delta_r}\,dm
	=\frac{3}{\pi^2}e^{-2r}.
\end{equation}
The claim follows from Lemma \ref{lemma:limit}.

\subsection{Proof of Theorem \ref{main1} part (2)}
A key ingredient in the proof of Theorem \ref{main1} part (2) is the following theorem of Kleinbock and Yu.

\begin{theorem}[\cite{KY}, Theorem 2.1]\label{KY}
Let $S$ be a bounded and measurable  subset of $\RR^2$. Denote 
$ S^-=  \{ x \in \RR^2: -x \in S \}$. If $S \cap S^- =\emptyset$ then 
\begin{equation}\label{eq:KY}
\| \widecheck{\1_S} \|_2^2= 
\frac{6}{\pi^2}  \left( m(S)+ \sum_{ n \neq 0 } \frac{\phi(|n|)}{|n|}
\int_S |\mathcal{I}_{(x,y)}^n|  \ dx  \right)
\end{equation}
where $\phi$ denotes Euler's totient function, $\mathcal{I}_{(x,y)}^n \subseteq \RR$
is defined by 
\[ \mathcal{I}_{(x,y)}^n= \left\{ t \in \RR : 
n  \left( \frac{-y}{x^2+ y^2}, \frac{x}{x^2+ y^2} \right) + t (x,y) \in S \right\}. \]
and $|\mathcal{I}_{(x,y)}^n|$ denotes the length of $\mathcal{I}_{(x,y)}^n$ with respect to the Lebesgue measure. 
\end{theorem}

\begin{lemma} \label{secondmoment}
For $r \in (- \frac{1}{2} \log 2, 0]$ we have
\begin{equation*}
\| \widecheck{ \1_{\Delta_r} } \|_2^2
= \frac{6}{\pi^2}  \left( \frac{5}{2}e^{-2r}-2 +4r - 4r^2 \right).
\end{equation*} 
\end{lemma}
\begin{proof}
To prove this lemma we invoke Theorem \ref{KY} with $S=\Delta_r$. Clearly $\Delta_r\cap \Delta_r^{-}=\emptyset$. Let $(x, y) \in \Delta_{r}$. We will show that $ \mathcal{I}_{(x,y)}^n$ can only have a non-zero length for $n=1$ and $n=-1$, and we will also determine the subsets of $\Delta_{r}$ consisting of those $(x, y)$ for which  $\mathcal{I}_{(x,y)}^1$ and $\mathcal{I}_{(x,y)}^{-1}$ are of positive length.  For the computations of this proof it will be convenient to set $s=e^{-r}$ and to write $\Delta:=\Delta_r=\Delta_{-\log s}$. Note that for $r \in (- \frac{1}{2} \log 2, 0]$ we have $s \in [1, \sqrt{ 2})$. The area
of $\Delta$ is $s^2/2$, meaning that we are left with computing the second term in \eqref{eq:KY}.

\begin{figure}[h]
	
\begin{center}

\begin{tikzpicture}[scale=1.5]
    \fill [magenta!20] (1.2, 0)--( 2.02, 0.8)--(2,0)--cycle;
    \draw [<->,thick] (0,2.2) node (yaxis) [above] {$y$}
        |- (2.4,0) node (xaxis) [right] {$x$};
    \draw[thick] (0,0) -- (-0.3,0);
    \draw (2,0) node (s)  [below] {\tiny $s$};
    \draw (0,0) -- (2,0) -- (2,2) -- (0,0);
    \draw [->] (0,0)--(1.65,0.33) node (v) [above] {\tiny $(x,y)$};
     \draw [->] (0,0)--(-0.12,0.5) node (vv) [left] {{\tiny{$(-y/x^2+y^2, x/x^2+ y^2)$}}};
    \draw [dashed] (-0.12,0.5) -- (2.58, 1.04);
    \draw (0.6, 0.6) node (A) [above] {\tiny $A_1$};
    \draw (2.2, 0.93) node (B) [right, above]  {\tiny $B_1$};
\draw [dashed, red] (1.2, 0)--( 2.02, 0.8);
\draw (1.82, 0.18) node (c) {\tiny $C_s$};
\end{tikzpicture}
\caption{}

\end{center}
\label{linetriangle}
\end{figure}

For a given $(x, y) \in \Delta$, and $n \in \ZZ$, let 
$\ell_{n}$ be the line given by the parametric equation 
\[ n  \left( \frac{-y}{x^2+ y^2}, \frac{x}{x^2+ y^2} \right) + t (x,y), \quad t \in \RR. \]
We first consider the case $n\geq 1$. In this case, note that $\ell_n$ is issued at a point in the second quadrant and has a positive slope. This implies that it cannot intersect the positive part of the $x$-axis.
 $ \mathcal{I}_{(x,y)}^n$ is defined by the intersection of $\ell_n$ and $\Delta$, a line and a triangle, and therefore it has to be a single interval, possibly a single point or the empty set.  We will compute the
values of $t$ that correspond to the intersection points between $\ell_n$ and the boundary of $\Delta$. We denote these two points by $A_n$ and $B_n$, see Figure 6. 
It thus follows that the intersection points $A_n$ and $B_n$, if they exist, are precisely the intersection of $\ell_n$ with lines $x=y$ and $ x=s$, respectively. Also, $\ell_n$ clearly intersects either both of these lines or neither. To determine $A_n$ we need to solve
\[  t_1 x- \frac{ny}{x^2+ y^2}=   t_1y + \frac{nx}{x^2+ y^2},  \]
which holds for
\[ t_1=  \frac{n(x+y)}{(x^2+ y^2)(x-y)}. \]
From here the coordinates of $A_n$ can be easily computed
\[ x_{A_n}= y_{A_n}= \frac{n}{x-y}. \]
Clearly for $A_n$ to lie on $\Delta$ we need $0 < x_{A_n} \le s$. Since $(x,y)\in\Delta$, the inequality $x_{A_n} > 0$ clearly implies that $n \ge 1$. Moreover, if $n \ge 2$ we have 
\[ \frac{n}{x-y} \ge \frac{2}{x-y} \ge \frac{2}{s} > s \]
where the last inequality uses the assumption $s< \sqrt{ 2} $. We conclude that for $n \ge 2$, the line $\ell_n$ does not intersect the triangle. Hence the only contribution to the second term of \eqref{eq:KY} comes from the $n=1$ case. Note that in this case Euler's totient function gives $\phi(1)=1$.  
Similarly, the coordinates of $B_1$ can be computed from $x_{B_1}=s$, that is,  
\[ x_{B_1}= t_2 x- \frac{y}{x^2+ y^2}=s \quad \overset{ }{\Rightarrow} \quad  t_2= \frac{s(x^2+ y^2)+y}{x(x^2+ y^2)}. \]
From here we have 
\[ |\mathcal{I}_{(x,y)}^1|= t_2- t_1 = \frac{s(x-y)-1}{x(x-y)}. \]
Hence we need 
\[ s(x-y)-1>0 \quad \overset{ }{\Rightarrow} \quad y< x-   \frac{1}{s}. \]
The points satisfying the last inequality are exactly those belonging to the triangle with vertices at $(s^{-1}, 0), (s, 0), (s, s- s^{-1}).$  This triangle, which we denote by $C_s$, is depicted in Figure 6. 
The above discussion shows that the length $|\mathcal{I}_{(x,y)}^1|$ is given by 
\[   |\mathcal{I}_{(x,y)}^1| = \begin{cases} \frac{s(x-y)-1}{x(x-y)}  & \textrm{if } \, (x, y) \in C_s\\
0 & \textrm{if } \, (x, y) \in \Delta \setminus C_s
\end{cases}    \]
Hence we have 
\begin{equation}
	\label{intn=1}
\begin{split}      
\int_{\Delta} |\mathcal{I}_{(x,y)}^n|  \ dx \, dy & = \int_{s^{-1}}^s \int_{y=0}^{y=x-s^{-1}}
  \left( \frac{s}{x}-\frac{1}{x(x-y)} \right) \,  dy \ dx \\
&=  \int_{s^{-1}}^s \frac{s}{x} ( x - s^{-1} ) - x^{-1} \left( \log x + \log s \right) 
\ dx \\
&=  \int_{s^{-1}}^s ( s-x^{-1}) \ dx  -  \int_{s^{-1}}^s  x^{-1}{\log x} dx 
 - \log s  \int_{s^{-1}}^s  x^{-1} \ dx \\
 &= s^2-1 -2 \log(s)- 2 (\log s)^2.
\end{split}
\end{equation}
where the last equality uses $J= \int_{s^{-1}}^s  x^{-1}{\log x} \, dx=0$, which can be easily verified by observing that the substitution $x \mapsto x^{-1}$ transforms $J$ to $-J$. 

The case $n\leq -1$ is argued along the same lines. The argument that no contribution is made to  \eqref{eq:KY} for $n\leq -2$ is identical. For $n=-1$ similar computations give that only points in the triangle with vertices $(\frac{1}{s},\frac{1}{s}), (s,s), (s,\frac{1}{s})$ give rise to non-empty intersection between $l_{-1}$ and $\Delta$. For such values of $(x,y)$, the value of $|\mathcal{I}_{(x,y)}^{-1}| $ is given by $\frac{sy-1}{xy}$ which integrated over $\Delta$ gives exactly the same contribution as in \eqref{intn=1}.

The claim will now follow from \eqref{eq:KY} with the change of variables 
$s= e^{-r}$. 
\end{proof}

\begin{lemma}\label{up}
Let $n \ge 2$. For $r >  \frac{1}{2} \log n$  (equivalently $s < \sqrt{n}$) and any unimodular lattice $ \Lambda \in \LLL_2$, the intersection $
\Lambda^\ast \cap \Delta_{r}$  has cardinality at most $n$. 
\end{lemma}

\begin{proof}
Again, set $\Delta:=\Delta_r=\Delta_{-\log s}$. Suppose for contradiction that $P_1, P_2,\dots, P_{n+1}$ denote $n+1$ distinct points in $\Lambda^\ast \cap \Delta$. After possibly re-enumerating them, we can assume that the $n$ triangles $OP_1P_2, OP_2P_3,\dots, OP_nP_{n+1}$ do not have common interior points. Since $\Lambda$ is unimodular, the area of each triangle is $1/2$ and the area of $\Delta$ must then be at least $\frac{n}{2}$. Hence $\frac{s^2}{2} \ge \frac{n}{2}$, which implies $s \ge \sqrt{n}$, a contradiction.  
\end{proof}
Note, in particular, that for $n=2$ the above lemma states that for $r >- \frac{1}{2} \log 2$, $\Lambda^\ast \cap \Delta_{r}$  has cardinality at most $2$.
Recall from \eqref{eq:ModCharDelta} that
\begin{equation}\label{eq:ModCharDelta2}
\widecheck{ \1_{\Delta_r} }(\Lambda)=\sum_{n=0}^{\infty}n\1_{\tub{\Lambda:\num{\Lambda^{\ast} \cap \Delta_r}=n}}.
\end{equation}
By Lemma \ref{up} we have $ |\Lambda^\ast \cap \Delta_{r}|$ can only take the values $k \in \{0, 1,2 \}$, hence \eqref{eq:ModCharDelta2} reduces to
\begin{equation}\label{eq:ModCharDelta3}
	\widecheck{ \1_{\Delta_r} }(\Lambda)=\1_{\tub{\Lambda:\num{\Lambda^{\ast} \cap \Delta_r}=1}}+2\cdot\1_{\tub{\Lambda:\num{\Lambda^{\ast} \cap \Delta_r}=2}}
\end{equation}

For $k \in \{ 0,1,2 \}$ set $\A_k=\A_k(r):=\tub{\Lambda:\num{\Lambda^{\ast}\cap \Delta_r}=k}$. By \eqref{firstmoment}, Lemma \ref{secondmoment} and \eqref{eq:ModCharDelta3} we have

\begin{equation}
\begin{split}      
\| \widecheck{\1_{\Delta_r}} \|_1 &= \mu( \A_1)+ 2 \mu(\A_2) =\frac{3}{\pi^2}e^{-2r}.\\
\| \widecheck{\1_{\Delta_r}} \|_2^2 &= \mu( \A_1)+ 4 \mu(\A_2)= \frac{6}{\pi^2}  \left( \frac{5}{2}e^{-2r}-2 +4r - 4r^2 \right). \\
\end{split}
\end{equation}
 
From here it is an easy computation to show that
\begin{equation*}
	\mu(\Hit^{\ast}(\Delta_r))=\mu(\A_1)+\mu(\A_2)=\frac{3}{\pi^2}\para{-e^{-2r}+4r^2-4r+2}.
\end{equation*}
The claim follows immediately from Lemma \ref{lemma:limit}.

\subsection{Proof of Theorem \ref{main1} part (3)}

\begin{proposition}\label{tail}
	There exist positive constants $C_0, C_1$  such that for all $r \le -\frac12 \log 2$ we have
	\[  C_0 e^{2r} \le     F(r) \le C_1  e^{2r}. \]
\end{proposition}
In fact, the constant $C_0$ may be chosen as $\frac{3}{4\pi}$. 
 
In the following let $\tB{R}:=B(0,R)\backslash\tub{0}$.

\begin{lemma}\label{CuspEstimate}
	For all $R< 1$ we have
	\begin{equation*}
	\mu(\Hit(\tB{R}))=\mu(\Hit^{\ast}(\tB{R}))=\frac{3}{\pi}R^2.
	\end{equation*} 
	\begin{proof} The first equality follows from Proposition \ref{hit}(\ref{Hit2}). We now prove the second equality. For $R<1$, $\tB{R}$ cannot contain two or more linearly independent primitive points since two such vectors would span a parallelogram of area less than 1. Hence if $\Lambda\in \Hit^{\ast}(\tB{R})$, then $\tB{R}$ contains exactly two points from $\Lambda^{\ast}$, say $v$ and $-v$. We conclude that 
		\begin{equation*}
		\Hit^{\ast}(\tB{R})=\tub{\Lambda: \num{\Lambda^{\ast}\cap \tB{R}}=2}.
		\end{equation*}
		By the same considerations as in Subsection \ref{MainProofFirstEnd} we see that
		\begin{equation*}
		\widecheck{\1_{\tB{R}}}=2\cdot \1_{\tub{\Lambda: \num{\Lambda^{\ast}\cap \tB{R}}=2}}=2\cdot \1_{\Hit^{\ast}(\tB{R})}.
		\end{equation*}
		Applying the primitive Siegel formula finally gives
		\begin{equation*}
		\mu(\Hit^{\ast}(\tB{R}))=\frac12 \frac{1}{\zeta(2)}m(\tB{R})=\frac{3}{\pi}R^2,
		\end{equation*}
		as claimed.
	\end{proof}
\end{lemma}

\begin{proof}[Proof of Proposition \ref{tail}]  As in earlier proofs we set $s=e^{-r}$, meaning $s\geq \sqrt{2}$, and $\Delta:=\Delta_r=\Delta_{-\log s}$. For a set $A\subset \RR$ we also write $\PMiss(A):=\tub{\Lambda:\Lambda^{\ast}\cap A=\emptyset}$.
	
	For the lower bound we write $\Delta'$ for the $90^{\circ}$ counterclockwise rotation of $\Delta$, see Figure 7. Recall that $F(r)=\mu(\PMiss(\Delta))$. Furthermore, $\mu(\PMiss(\Delta))=\mu(\PMiss(\Delta'))$ due to Lemma \ref{SLnRinv} and the fact that rotations belong to $\SL_2(\RR)$.
	The idea is now that since
	\begin{align*}
		F(r)&=\frac12\para{\mu(\PMiss(\Delta))+\mu(\PMiss(\Delta'))}\geq \frac12\para{\mu(\PMiss(\Delta)\cup\PMiss(\Delta'))},
	\end{align*} 
	we will obtain a lower bound by finding a subset of $\PMiss(\Delta)\cup\PMiss(\Delta')$ whose measure we can compute. Using the fact that a unimodular lattice cannot contain two or more short primitive vectors we will argue that
	\begin{equation}\label{HitMissInclusion}
		\PHit\para{\tB{\frac{1}{\sqrt{2}s}}}\subset \PMiss(\Delta)\cup\PMiss(\Delta').
	\end{equation}

	\begin{center}
		\begin{figure}
		\begin{tikzpicture}
		\draw (6,0) -- (6,3) -- (3,0) -- (6,0); 
		\draw (3,0) -- (3,3) -- (0,3) -- (3,0); 
		\draw (-1,0) -- (7,0); 
		\draw (3,-1) -- (3,4); 
		\draw (5.5,0.8)  node[anchor=east] {\footnotesize{$\Delta$}};
		\draw (2.5,2.1)  node[anchor=east] {\footnotesize{$\Delta'$}};
		\end{tikzpicture}
		\caption{}
		\end{figure}
	\end{center}
	
	Let $ \Lambda \in \PHit\para{\tB{\frac{1}{\sqrt{2}s}}}$ and $v \in \Lambda \setminus \{ 0 \}$ be a vector of least norm, i.e. $\norm{v}< \frac{1}{\sqrt{2}s}$. Clearly $v\in\Lambda^{\ast}$. For any other $w\in \Lambda^{\ast}$ we have that $\norm{v}\norm{w}\geq 1$, i.e. $\norm{w}\geq \frac{1}{\norm{v}}\geq \sqrt{2}s$. We conclude that if $v\in\Delta$ then $\Lambda\in\PMiss(\Delta')$ and if $v\notin \Delta$ then $\Lambda\in \PMiss(\Delta)$. This proves \eqref{HitMissInclusion} and gives the lower bound
	\begin{equation*}
		F(r)\geq \frac12 \para{\frac{3}{\pi}\para{\frac{1}{\sqrt{2}s}}^2}=\frac{3}{4\pi}e^{2r}
	\end{equation*}
	via Lemma \ref{CuspEstimate}, since $\frac{1}{\sqrt{2}s}\leq \frac12<1$ for $s\geq \sqrt{2}$. Setting $C_0:=3/4\pi$ concludes the lower bound.
	
	For the upper bound we invoke the results of Minkowski and Schmidt when $n=2$ and $S=\Delta$. Lemma \ref{Minkowski} gives
	\begin{equation}\label{MinkowskiDim2}
	\lambda_1(\Lambda)\lambda_2(\Lambda)\leq \frac{4}{\pi}.
	\end{equation}
	In Lemma \ref{Davenport} we will set $R=2s$. Then, since $s>1$, by Minkowski's convex body theorem \cite[Theorem 10]{Siegel} we have $ \lambda_1  \le 2 \le R$ as required. Recall that $F(r)=\mu(\Miss(\Delta))$ and assume that $\Lambda\in\Miss(\Delta)$. Then Lemma \ref{Davenport} gives  
	\[  \frac{s^2}{2}\le 2C_2 \lambda_2( \Lambda)  s. \]
	This implies that $ \lambda_2( \Lambda) \ge  \frac{s}{4C_2}$. By equation \eqref{MinkowskiDim2} we have  
	$ \lambda_1( \Lambda) \le 16 C_2/\pi s$, and hence $\Lambda\in \Hit\para{{\tB{\frac{16C_2}{\pi s}}}}$

	Hence $\Miss(\Delta)\subset  \Hit\para{{\tB{\frac{16C_2}{\pi s}}}}$. Taking measures and using Lemma \ref{CuspEstimate} again, we obtain
	\begin{equation*}
		F(r) \leq \mu\para{\Hit\para{{\tB{\frac{16C_2}{\pi s}}}}} 
		= \frac{3\cdot 16^2 C_2^2}{\pi^3} e^{2r}.
	\end{equation*}
	Setting $C_1:=\frac{3\cdot 16^2 C_2^2}{\pi^3}$ concludes the upper bound.
\end{proof}

\bibliographystyle{amsalpha}
\bibliography{homdyn.bib}

\end{document}